\documentclass{amsart}
\usepackage{amscd,amssymb,hyperref}

\newcommand{\parl}{{\rm(}}
\newcommand{\parr}{{\rm)}}

\renewcommand{\phi}{\varphi}
\renewcommand{\epsilon}{\varepsilon}
\renewcommand{\emptyset}{\varnothing}

\newcommand{\bE}{\mathbf E}
\newcommand{\bG}{\mathbf G}
\newcommand{\scG}{\widetilde{\bG}}
\newcommand{\bH}{\mathbf H}

\newcommand{\bP}{\mathbf P}
\newcommand{\bU}{\mathbf U}
\newcommand{\bZ}{\mathbf Z}

\newcommand{\cE}{\mathcal E}
\newcommand{\cF}{\mathcal F}
\newcommand{\cG}{\mathcal G}
\newcommand{\cH}{\mathcal H}

\newcommand{\cO}{\mathcal O}
\newcommand{\cP}{\mathcal P}
\newcommand{\cT}{\mathcal T}

\newcommand{\cZ}{\mathcal Z}

\renewcommand{\P}{\mathbb P}
\newcommand{\A}{\mathbb A}
\newcommand{\Z}{\mathbb Z}
\newcommand{\Q}{\mathbb Q}

\newcommand{\Aff}{\mathrm{Aff}}
\newcommand{\Et}{\mathrm{\acute Et}}
\DeclareMathOperator{\Id}{Id}

\DeclareMathOperator{\Out}{\mathbf{Out}}
\DeclareMathOperator{\Aut}{\mathbf{Aut}}
\DeclareMathOperator{\Orb}{Orb}
\DeclareMathOperator{\spec}{Spec}
\DeclareMathOperator{\Hom}{Hom}
\DeclareMathOperator{\GL}{GL}

\DeclareMathOperator{\Pic}{Pic}
\DeclareMathOperator{\Gr}{Gr}
\DeclareMathOperator{\Res}{Res}

\newcommand{\Gm}[1]{\mathop{\mathbb G_{m,#1}}}

\newcommand{\ad}{\mathrm{ad}}
\renewcommand{\sc}{\mathrm{sc}}
\newcommand{\spl}{\mathrm{spl}}

\theoremstyle{plain}
\newtheorem{theorem}{Theorem}

\newtheorem{proposition}{Proposition}[section]
\newtheorem{lemma}[proposition]{Lemma}
\newtheorem*{corollary*}{Corollary}

\theoremstyle{definition}
\newtheorem{definition}[proposition]{Definition}

\theoremstyle{remark}
\newtheorem{remark}[proposition]{Remark}
\newtheorem{remarks}[proposition]{Remarks}
\newtheorem*{remark*}{Remark}

\begin{document}

\title[On the Grothendieck--Serre Conjecture]{On the Grothendieck--Serre Conjecture about principal bundles and its generalizations}


\begin{abstract}
Let $U$ be a regular connected affine semi-local scheme over a field $k$. Let~$\bG$ be a reductive group scheme over $U$. Assuming that $\bG$ has an appropriate parabolic subgroup scheme, we prove the following statement. Given an affine $k$-scheme $W$, a principal $\bG$-bundle over $W\times_kU$ is trivial if it is trivial over the generic fiber of the projection $W\times_kU\to U$.

We also simplify the proof of the Grothendieck--Serre conjecture: let $U$ be a regular connected affine semi-local scheme over a field $k$. Let~$\bG$ be a reductive group scheme over $U$. A principal $\bG$-bundle over $U$ is trivial if it is trivial over the generic point of $U$.

We generalize some other related results from the simple simply-connected case to the case of arbitrary reductive group schemes.
\end{abstract}

\author{Roman Fedorov}
\email{fedorov@pitt.edu}
\address{University of Pittsburgh, PA, USA}

\maketitle

\section{Introduction and main results}
The conjecture of Grothendieck and Serre on principal bundles asserts that if $\bG$ is a reductive group scheme over a regular affine semi-local scheme $U$ and $\cE$ is a rationally trivial principal $\bG$-bundle over $U$, then $\cE$ is trivial. We refer the reader to Section~\ref{sect:DefConv} for the precise definitions. The conjecture has been proved in the case, when $U$ is a scheme over a field $k$ (see~\cite{FedorovPanin},~\cite{PaninFiniteFieldsIzvestiya}).

One of the main goals of this paper is to generalize this result to families as we now explain. Let $U$ and $\bG$ be as before and denote the generic point of $U$ by $\Omega$. Let $W$ be an affine $k$-scheme. Then a principal $\bG$-bundle $\cE$ over $W\times_kU$ is trivial, provided its restriction to $W\times_k\Omega$ is trivial, $\bG$ satisfies some isotropy condition, and $U$ is regular over $k$.

We note that our result is~\cite[Thm.~1.1]{PaninStavrovaVavilov} and~\cite[Thm.~7.1]{PaninNiceTriples}, provided that our group scheme is isotropic, simple, and simply-connected, and $U$ is the spectrum of a semi-local ring of finitely many closed points on an irreducible smooth affine $k$-variety.

In this paper, we will also give a streamlined and simplified proof of the conjecture of Grothendieck and Serre.

\subsection{Fully reducible reductive group schemes} We start with formulating precisely the isotropy condition mentioned above. Let $\bG$ be a reductive group scheme over a connected scheme $U$. Let $\bZ$ be the center of $\bG$. The adjoint group scheme $\bG^{\ad}:=\bG/\bZ$ is semisimple (see~\cite[Exp.~XXII, Prop.~4.3.4(i), Prop.~4.3.5(ii), Def.~4.3.6]{SGA3-3}). By~\cite[Exp.~XXIV, Prop.~5.10]{SGA3-3} there is a sequence $U_1,\ldots,U_r$ of finite \'etale connected $U$-schemes such that
\begin{equation}\label{eq:prod}
    \bG^{\ad}\simeq\prod_{i=1}^r\bG^i,
\end{equation}
where $\bG^i$ is the Weil restriction of a simple $U_i$-group scheme. Note that the group schemes $\bG^i$ are uniquely defined by $\bG$ up to isomorphism.

\begin{definition}
We say that a reductive $U$-group scheme $\bG$ is \emph{strongly locally isotropic\/ if each factor $\bG^i$ of $\bG^{\ad}$ is isotropic Zariski locally over $U$.}
\end{definition}

\begin{remarks}
(i) If $\bG$ is a simple group scheme over $U$ (or more generally, is the Weil restriction of a simple group scheme via a finite \'etale morphism $U'\to U$ with connected~$U'$ and $U$), then it is strongly locally isotropic if and only if Zariski locally over $U$ it contains a proper parabolic subgroup scheme; see~\cite[Exp.~XXVI, Cor.~6.14]{SGA3-3}.

(ii) It follows form the previous remark that if a reductive groups scheme is strongly locally isotropic, then it is locally isotropic.

(iii) Equivalently, one can show that a group scheme $\bG$ is strongly locally isotropic if and only if Zariski locally over $U$ it contains a parabolic subgroup scheme whose image in any non-abelian quotient of $\bG$ is a proper subgroup scheme of the quotient.
\end{remarks}

\subsection{The Grothendieck--Serre conjecture for families}
Here is our first main result.

\begin{theorem}\label{th:main}
Let $U$ be a regular connected affine semi-local scheme over a~field $k$. Denote by $\Omega$ the generic point of $U$. Let $\bG$ be a strongly locally isotropic reductive group scheme over~$U$. Let $W$ be an affine $k$-scheme. Let $\cE$ be a principal $\bG$-bundle over $W\times_k U$. If the restriction of $\cE$ to $W\times_k\Omega$ is trivial, then $\cE$ is trivial.
\end{theorem}

This theorem will be proved in Section~\ref{sect:ProofOfMain}. It is known that the requirement that $\bG$ be strongly locally isotropic is necessary (see counterexamples in~\cite[Sect.~2.3]{FedorovExotic}). However, we conjecture that a similar statement is true even when $U$ is not a scheme over a field (that is, in the mixed characteristic case).

\subsection{A simplified proof of the Grothendieck--Serre Conjecture}
We will also present a simplified proof of the Grothendieck--Serre conjecture in Section~\ref{sect:ProofOfGrSerre}. Precisely, we will re-prove the following theorem.
\begin{theorem}[\cite{FedorovPanin},~\cite{PaninFiniteFieldsIzvestiya}]\label{th:GrSerre}
   Let $U$ be a regular connected affine semi-local scheme over a~field. Let $\Omega$ be the generic point of $U$. Let $\bG$ be a reductive group scheme over~$U$. Let $\cE$ be a principal $\bG$-bundle over $U$. If the restriction of $\cE$ to $\Omega$ is trivial, then $\cE$ is trivial.
\end{theorem}
Theorem~\ref{th:GrSerre} is derived from Theorem~\ref{th:FP} (the ``Section Theorem'') below using the results of~\cite{PaninNiceTriples}. Theorem~\ref{th:FP} was only known before for simple simply-connected group schemes. Thus, to prove Theorem~\ref{th:GrSerre}, one had first to reduce to the simple simply-connected case, using the so-called purity theorems~\cite{PaninPurity2010,PaninPurity17}. In this paper, we will show that this Section Theorem holds for all reductive group schemes, thus eliminating the difficult reduction to the simple simply-connected case. We will outline the strategy of the proof of Theorem~\ref{th:FP} after its formulation in Section~\ref{sect:SectTh}.

In the case, when $\bG$ is a torus, the Grothendieck--Serre conjecture was settled by J.-L.~Colliot-Th\'{e}l\`{e}ne and J.-J.~Sansuc in~\cite{ColliotTheleneSansuc}. It seems that our proof is new even in this case.

\subsection{An application: principal bundles over affine spaces}
The following theorem is a generalization of~\cite[Cor.~1.7]{PaninStavrovaVavilov}.
\begin{theorem}
Let $U$ be a regular connected affine scheme over $\Q$ and let $\bG$ be a~strongly locally isotropic reductive group scheme over $U$. Let $n$ be a non-negative integer, and let~$\cE$ be a~principal $\bG$-bundle over the affine space $\A_U^n$ whose restriction to the origin $U\times 0\subset\A_U^n$ is trivial. Then $\cE$ is trivial.
\end{theorem}
\begin{proof}
The proof is by induction on $n$. The case $n=0$ is obvious. Assume that the theorem is proved for $n-1$. Let $\cE$ be a principal $\bG$-bundle over $\A_U^n$. Write $\A_U^n=\A_U^{n-1}\times_U\A_U^1$. Let $H$ be the zero section $\A_U^{n-1}\times 0$ so that we identify $\A^n_U=\A_H^1$. Note that $H$ is integral. Let $\Omega$ be the generic point of $H$. We have a commutative diagram
\[
\begin{CD}
\Omega @>>>\A^1_\Omega \\
@VVV @VVV\\
H   @>>> \A^1_H=\A_U^n,
\end{CD}
\]
where the horizontal arrows are embeddings of the zero sections. By induction hypothesis the restriction of $\cE$ to $H$ is trivial, so its restriction to $\Omega$ is trivial as well. Since the restriction of $\cE$ to $\Omega$ is trivial, its restriction to $\A^1_\Omega$ is also trivial by Raghunathan--Ramanathan Theorem (see~\cite{RagunathanRamanthan,GilleTorseurs}, we are using that $U$ has characteristic zero).

Next, let $\xi$ be any point of $H$ and let $W$ be the spectrum of $\cO_{H,\xi}$. The restriction of $\cE$ to $\A^1_W$ via the obvious morphism is trivial by our Theorem~\ref{th:main}, since it is trivial over $\A^1_\Omega$. Further, $U$ is normal so, according to \cite[Cor.~3.2]{ThomasonResolution}, we can embed $\bG$ into $\GL_{n,U}$ for some $n$. Thus we can apply~\cite[Korollar~3.5.2]{MoserGrSerre} to see that the principal $\bG$-bundle $\cE$ is trivial over $\A^1_H=\A_U^n$.
\end{proof}

\subsection{Section Theorems}\label{sect:SectTh} The following Section Theorem will be used in the proof of Theorem~\ref{th:GrSerre} in Section~\ref{sect:ProofOfGrSerre}.

\begin{theorem}\label{th:FP}
Let $U$ be an affine semi-local scheme. Assume that either $U$ is a scheme over an infinite field, or $U$ is a scheme over a finite field and the residue fields of all the closed points of $U$ are finite. Let~$\bG$ be a reductive group scheme over $U$. Assume that $Z$ is a closed subscheme of $\A^1_U$ finite over~$U$. Let~$\cE$ be a principal $\bG$-bundle over $\A^1_U$ trivial over $\A^1_U-Z$. Then for every section $\Delta\colon U\to\A^1_U$ of the projection $\A^1_U\to U$ the principal $\bG$-bundle $\Delta^*\cE$ is trivial.
\end{theorem}
This is a generalization of~\cite[Thm.~2]{FedorovPanin} and of~\cite[Thm.~1.6]{PaninFiniteFieldsIzvestiya} from simple simply-connected to reductive group schemes. This theorem will be proved in Section~\ref{sect:ProofFP}.

For not necessarily semi-local $U$ we have a weaker statement, which will be used in Section~\ref{sect:ProofOfMain} to prove Theorem~\ref{th:main}.

\begin{theorem}\label{th:FP2}
Let $U$ be an affine Noetherian connected scheme over a field. Let~$\bG$ be a reductive group scheme over $U$ such that $\bG$ can be embedded into $\GL_{n,U}$ for some~$n$. Assume that $\bG$ is strongly locally isotropic. Assume that $Z\subset\A^1_U$ is a closed subscheme finite over $U$. Let $\cE$ be a principal $\bG$-bundle over $\A^1_U$ trivial over $\A^1_U-Z$. Then for every section $\Delta\colon U\to\A^1_U$ of the projection $\A^1_U\to U$ the principal $\bG$-bundle $\Delta^*\cE$ is trivial.
\end{theorem}

This Section Theorem will be proved in Section~\ref{sect:ProofFP2}.

\begin{remark}
The condition that $\bG$ can be embedded into $\GL_{n,U}$ for some $n$ is satisfied in many cases: e.g.\ if $\bG$ is semisimple or if $U$ is normal, see~\cite[Cor.~3.2]{ThomasonResolution}.
\end{remark}

The idea of the proofs of the Section Theorems above is the following: first, we extend the principal $\bG$-bundle $\cE$ to a principal $\bG$-bundle $\hat\cE$ over $\P^1_U$. If $\bG$ is not simply-connected, then the usual proof goes through with some modifications, provided that the restrictions of~$\cE$ to the closed fibers of $\P^1_U\to U$ are in the neutral connected component of the stack of principal bundles. This can always be achieved by pulling back $\hat\cE$ via a cover $\P^1_U\to\P^1_U$ of a sufficiently divisible degree.

\subsection{Definitions, conventions, and notation}\label{sect:DefConv}
All rings in this paper are commutative and unital. A semi-local ring is a Noetherian ring having only finitely many maximal ideals. An \emph{affine semi-local scheme\/} is a scheme isomorphic to the spectrum of a semi-local ring.

A group scheme $\bG$ over a scheme $U$ is called \emph{reductive\/} if $\bG$ is affine and smooth as a~$U$-scheme and, moreover, the geometric fibers of $\bG$ are connected reductive algebraic groups (see~\cite[Exp.~XIX, Definition~2.7]{SGA3-3}). A smooth group scheme over a field $k$ is called \emph{a $k$-group}.

A $U$-scheme $\cE$ with a left action $act\colon\bG\times\cE\to\cE$ is called \emph{a principal $\bG$-bundle over $U$} if~$\cE$ is faithfully flat and quasi-compact over $U$ and the action is simply transitive, that is, the morphism $(act,p_2)\colon\bG\times_U\cE\to\cE\times_U\cE$ is an isomorphism (see~\cite[Section~6]{GrothendieckFGA}). A principal $\bG$-bundle $\cE$ over $U$ is \emph{trivial} if $\cE$ is isomorphic to $\bG$ as a~$U$-scheme with an action of $\bG$. This is well-known to be equivalent to the projection $\cE\to U$ having a section. We will use the term ``principal $\bG$-bundle over $T$'' to mean a principal $\bG_T$-bundle over $T$. We usually drop the adjective `principal'.

A subgroup scheme $\bP\subset\bG$ is \emph{parabolic\/} if $\bP$ is smooth over $U$ and for all geometric points $\spec k\to U$ the quotient $\bG_k/\bP_k$ is proper over $k$ (here $k$ is an algebraically closed field). This coincides with~\cite[Exp.~XXVI, Def.~1.1]{SGA3-3}.

\subsection{Acknowledgements} The first draft of this paper was written during the author's stay at the Universit\'e
Claude Bernard Lyon~1 for the quarter ``Algebraic groups and geometry of the Langlands program''. He wants to thank the organizers and especially Philippe Gille for his interest and stimulating discussions. The author is grateful to Dima Arinkin, Vladimir Chernousov, and Ivan Panin for constant interest in his work. The author is partially supported by the NSF DMS grant 1764391. A~part of the work was done during the author's stay at Max Planck Institute for Mathematics in Bonn. The author is thankful to the anonymous referee for reading the paper very carefully.

\section{Proofs of Theorem~\ref{th:FP} and of Theorem~\ref{th:FP2}}
We need some preliminaries.

\subsection{Topologically trivial principal bundles over $\P^1$}
Let $G$ be a semisimple group scheme over a field $k$. Let $\phi\colon G^{\,\sc}\to G$ be the simply-connected central cover. In other words, $G^{\,\sc}$ is simply-connected and $\phi$ is a central isogeny (in particular, $\phi$ is finite and flat).

\begin{definition}\label{def:TopTriv}
A Zariski locally trivial $G$-bundle $E$ over $\P^1_k$ is called \emph{topologically trivial} if it can be lifted to a Zariski locally trivial $G^{\,\sc}$-bundle. More precisely, this means that there is a Zariski locally trivial $G^{\,\sc}$-bundle $E^{\sc}$ over $\P^1_k$ such that $\phi_*E^{\sc}\simeq E$.
\end{definition}

\begin{remark}
If $k$ is the field of complex numbers, then a principal bundle over $\P^1_k$ is topologically trivial in the sense of Definition~\ref{def:TopTriv} if and only if it is topologically trivial in the usual sense, that is, it has a continuous section (cf.~\cite[Cor.~4.1.2]{SorgerLecturesBundles}), which justifies the name.
\end{remark}

We need the following proposition.
\begin{proposition}\label{pr:PullbackTopTriv}
For every Zariski locally trivial $G$-bundle $E$ over $\P^1_k$ and for every finite morphism $\psi\colon\P^1_k\to\P^1_k$ whose degree is divisible by the degree of $\phi$, the $G$-bundle $\psi^*E$ is topologically trivial.
\end{proposition}
Before giving the proof of the proposition we recall the description of Zariski locally trivial $G$-bundles over $\P^1_k$. Let $T\subset G$ be a maximal split torus of $G$. Let~$E$ be a Zariski locally trivial $G$-bundle over $\P^1_k$. Then by~\cite[Thm.~3.8(b)]{GilleTorseurs}, there is a co-character $\lambda\colon\Gm{k}\to T$ such that $E\simeq\lambda_*\cO(1)^\times$. Here $\cO(1)$ is the hyperplane line bundle over $\P^1_k$; the $\Gm{k}$-bundle $\cO(1)^\times$ is the complement of the zero section in $\cO(1)$. We are slightly abusing the notation, denoting the composition $\Gm{k}\xrightarrow\lambda T\hookrightarrow G$ by $\lambda$ as well.

\begin{proof}[Proof of Proposition~\ref{pr:PullbackTopTriv}]
Put $d:=\deg\phi$. Let $T^{\,\sc}$ be a maximal split torus of $G^{\,\sc}$. By~\cite[Thm.~2.20(ii)]{Borel-Tits2}, $T:=\phi(T^{\,\sc})$ is a maximal split torus of $G$. The $k$-group scheme $T\times_{G}G^{\,\sc}$ is of multiplicative type by~\cite[Exp.~XVII, Prop.~7.1.1(b)]{SGA3-2} and the isogeny $T\times_{G}G^{\,\sc}\to T$ also has degree $d$. It is clear that $T^{\,\sc}$ is the toral part of $T\times_{G}G^{\,\sc}$. It is also clear that $\phi|_{T^{\,\sc}}\colon T^{\,\sc}\to T$ is an isogeny whose degree divides $d$ (indeed, we can check it over an algebraic closure of $k$ in which case we may assume that $T\times_GG^{\,\sc}$ is diagonalizable).

Denote the degree of the isogeny $\phi|_{T^{\,\sc}}\colon T^{\,\sc}\to T$ by $d'$. It is also the index of the co-character lattice $X_*(T^{\,\sc})$ in $X_*(T)$. Let $E$ be a Zariski locally trivial $G$-bundle over $\P^1_k$. As we have already mentioned, by~\cite[Thm.~3.8(b)]{GilleTorseurs} there is a co-character $\lambda\colon\Gm{k}\to T$ such that $E\simeq\lambda_*\cO(1)^\times$. Let $\psi\colon\P^1_k\to\P^1_k$ be a finite morphism of degree $n$. Then
\[
    \psi^*E\simeq\lambda_*\cO(n)^\times\simeq(n\lambda)_*\cO(1)^\times,
\]
where $\cO(n)$ is the $n$-th tensor power of $\cO(1)$. If $d$ divides $n$, then $d'$ divides $n$ as well, so $n\lambda$ is a co-character of $X_*(T^{\,\sc})$ and it is clear that $\psi^*E$ can be lifted to a~$G^{\,\sc}$-bundle. The Proposition~\ref{pr:PullbackTopTriv} is proved.
\end{proof}

It is clear from the proof, that it is enough to require that the degree of $\psi$ is divisible by the exponent of the kernel of $\phi$.

\subsection{Recollection on affine Grassmannians}\label{sect:AffGr} We will use affine Grassmannians of group schemes defined in~\cite{FedorovExotic} in the proof of Theorem~\ref{th:MainThm2} below. We only consider the affine Grassmannians for semisimple group schemes. The results below should hold in bigger generality, for example, if the group scheme is reductive and can be embedded into the general linear group scheme. Since we are not aware of a reference, we will restrict ourselves to the semisimple case.

For an affine scheme $T=\spec S$, put $D_T:=\spec S[[t]]$ and $\dot D_T:=\spec S((t))$, where $S((t)):=S[[t]](t^{-1})$.

Recall the definition of affine Grassmannians from~\cite[Sect.~5.1]{FedorovExotic}. Consider a~connected affine scheme $U=\spec R$; let $\Aff/U$ be the (big) \'etale site of affine schemes over $U$ and $\Et/U$ be the (big) \'etale site of schemes over $U$. Recall that a~\emph{$U$-space} is a sheaf of sets on $\Et/U$. We can equivalently view it as a sheaf on $\Aff/U$ (see~\cite[Exp.~VII, Prop.~3.1]{SGA4-2}). Let $\bG$ be a smooth affine $U$-group scheme. The \emph{affine Grassmannian\/} $\Gr_\bG$ is defined as the sheafification of the presheaf, sending an affine $U$-scheme $T$ to the set $\bG(\dot D_T)/\bG(D_T)$. (The morphism $\dot D_T\to D_T$ induces a~morphism $\bG(D_T)\to\bG(\dot D_T)$. It is obvious that this morphism is injective and we identify $\bG(D_T)$ with its image.) If $\bG$ is semisimple, then $\Gr_\bG$ is an inductive limit of schemes over $U$ (see~\cite[Prop.~5.11]{FedorovExotic}). These schemes may be chosen projective over~$U$, though we will not use it.

Let $Y$ be a finite and \'etale over $U$ subscheme of $\A^1_U$ (automatically closed). Assume also that $Y\ne\emptyset$, then the projection $Y\to U$ is surjective (being both open and closed). Let $\cE$ be a $\bG$-bundle over $\P^1_U$. A~\emph{modification\/} of $\cE$ at $Y$ is a pair $(\cF,\tau)$, where $\cF$ is a $\bG$-bundle over $\P_U^1$ and $\tau$ is an isomorphism
\[
    \cF|_{\P^1_U-Y}\xrightarrow{\tau}\cE|_{\P^1_U-Y}
\]
(cf.~\cite[Sect.~7.3]{FedorovExotic}). We have an obvious notion of an isomorphism of modifications of $\cE$ at $Y$.

Fix a $\bG$-bundle $\cE$ over $\P^1_U$ and assume that it is trivial in a Zariski neighbourhood of $Y\subset\A^1_U$. Fix such a trivialization $\sigma$. Let $\Psi_\sigma$ be the functor, sending a $U$-scheme $T$ to the set of isomorphism classes of modifications of $\cE|_{\P^1_T}$ at $Y\times_UT$. Recall~\cite[Prop.~7.5]{FedorovExotic}:

\begin{proposition}\label{pr:ModAffGr}
The functor $\Psi_\sigma$ is canonically isomorphic to the functor sending a $U$-scheme $T$ to $\Gr_\bG(Y\times_UT)$.
\end{proposition}

Note that this isomorphism depends on the trivialization $\sigma$ of $\cE$ in a neighborhood of $Y$. Let $\sigma'$ be another trivialization on a (possibly different) Zariski neighborhood of $Y$. The restrictions of $\sigma$ and $\sigma'$ to the formal neighborhood of $Y$ differ by a jet $\alpha\in L^+\bG(Y)$, where the jet group scheme $L^+\bG$ represents the functor $T\mapsto\bG(D_T)$. Note that $L^+\bG$ acts on $\Gr_\bG$. The proof of the following lemma is clear from the proof of~\cite[Prop.~5.1]{FedorovExotic}.

\begin{lemma}\label{lm:ChangeOfTriv}
    The functors $\Psi_{\sigma'}$ and $\Psi_\sigma\circ\tilde\alpha$ are canonically isomorphic, where $\tilde\alpha$ stands for the automorphism of $\Gr_\bG$ given by the action of $\alpha$.
\end{lemma}

\begin{remarks}
To identify the modifications with sections of affine Grassmannian, it is enough to trivialize $\cE$ on a formal neighborhood of $Y$. Such a trivialization exists if and only if $\cE|_Y$ is trivial (because $\cE$ is smooth over $\P^1_U$). If $\cE$ is not trivial on $Y$, then the modifications are parameterized by a twist of the affine Grassmannian.
\end{remarks}

The unit section of $\bG$ gives rise to a unit section $\Id_{\Gr}\in\Gr_\bG(Y)$. This section corresponds to the trivial modification $(\cE,\Id_{\cE}|_{\P_U^1-Y})$ under the above isomorphism.

It is clear that we have a natural isomorphism $\Gr_{\bG_1\times_U\bG_2}=\Gr_{\bG_1}\times_U\Gr_{\bG_2}$.

Note that there is a canonical automorphism of $\P^1_U$ switching $\P^1_U-(U\times0)$ and $\A^1_U$. We use this automorphism to identify points of $\Gr_\bG(U)$ with modifications of the trivial $\bG$-bundle at $U\times\infty$, that is, with pairs $(\cE,\tau)$, where $\cE$ is a $\bG$-bundle over $\P^1_U$, $\tau$ is a trivialization of $\cE$ over $\A^1_U$.

The following is a slight generalization of~\cite[Prop.~7.1]{FedorovExotic}.
\begin{lemma}\label{lm:ExtendSect}
Let $Y$ be a connected affine scheme; let $y_1,\ldots,y_n\in Y$ be closed points. Let $\bH$ be the Weil restriction of a simple simply-connected $Y'$-group scheme $\bH'$ via a finite \'etale morphism $Y'\to Y$ with a connected $Y'$, and assume that $\bH$ contains a proper parabolic subgroup scheme. Then the restriction morphism
\[
    \Gr_\bH(Y)\to\prod_{i=1}^n\Gr_\bH(y_i)
\]
is surjective.
\end{lemma}
\begin{proof}
If $\cE'$ is an $\bH'$-bundle over $\P^1_{Y'}$, then its Weil restriction $\Res_{Y'/Y}(\cE')$ is an $\bH$-bundle over $\P^1_Y$, and a trivialization of $\cE'$ over $\A^1_{Y'}$ gives rise to a trivialization of $\Res_{Y'/Y}(\cE')$ over $\A^1_Y$. Thus we get a map $\Gr_{\bH'}(Y')\to\Gr_\bH(Y)$. It is easy to check that this map is a bijection (cf.~\cite[Exp.~XXIV, Prop.~8.2]{SGA3-3}).

Therefore, denoting by $\{y'_1,\ldots,y'_m\}$ the preimage in $Y'$ of the set $\{y_1,\ldots,y_n\}$,  we get a commutative diagram
\[
\begin{CD}
\Gr_{\bH'}(Y')@>>>\prod_{i=1}^m\Gr_{\bH'}(y'_i)\\
@V\simeq VV @V\simeq VV\\
\Gr_\bH(Y)@>>>\prod_{i=1}^n\Gr_\bH(y_i).
\end{CD}
\]

\emph{Claim. $\bH'$ has a proper parabolic subgroup scheme.} Indeed, let $\cP'$ be the variety of parabolic subgroup schemes of $\bH'$, let $\cP$ be the variety of parabolic subgroup schemes of $\bH$ (cf.~\cite[Exp.~XXVI, Cor.~3.5]{SGA3-3}). Since a parabolic subgroup scheme of a reductive group scheme gives rise to a parabolic subgroup scheme of its Weil restriction, we get a morphism $\Res_{Y'/Y}\cP'\to\cP$, where $\Res$ is the Weil restriction functor. Let us check that this morphism is an isomorphism. Since this is enough to check after an \'etale base change, we may assume that $Y'$ is a disjoint union of $n$ copies of $Y$ for an integer $n$ and that $\bH'$ is split. Then $\bH$ is the product of $n$ copies of $\bH'$ and the statement reduces to the statement that the parabolic subgroup schemes of $\bH\simeq\prod_{i=1}^n\bH'$ are exactly subgroup schemes of the form $\bP_1\times\ldots\times\bP_n$, where each $\bP_i$ is a parabolic subgroup scheme of $\bH$. Choose a split maximal torus and a Borel subgroup scheme in $\bH'$, we get a split maximal torus and a Borel subgroup scheme in $\bH\simeq\prod_{i=1}^n\bH'$. Then we have a notion of standard parabolic subgroup schemes, and it is clear that every standard parabolic subgroup scheme in $\bH$ is the product of standard parabolic subgroup schemes in $\bH'$. It remains to note that every parabolic subgroup scheme is locally conjugate to a standard one.

Thus every parabolic subgroup scheme of $\bH$ gives rise to a parabolic subgroup scheme of $\bH'$. It is clear that a proper parabolic subgroup scheme of $\bH$ gives rise to a proper parabolic subgroup scheme of $\bH'$. This proves the claim.

Thus, replacing $\bH$ with $\bH'$, we may assume from the beginning that $\bH$ is a simple simply-connected group scheme. The rest is very similar to the proof of~\cite[Prop.~7.1]{FedorovExotic} but we give a proof for the sake of completeness. Let $\bP^+$ be a proper parabolic subgroup scheme of $\bH$. Since $Y$ is an affine scheme, by~\cite[Exp.~XXVI, Cor.~2.3, Thm.~4.3.2(a)]{SGA3-3}, there is an opposite to $\bP^+$ parabolic subgroup scheme $\bP^-\subset\bH$. Let $\bU^+$ be the unipotent radical of $\bP^+$, and let $\bU^-$ be the unipotent radical of $\bP^-$. We will write $\bE$ for the functor, sending a $Y$-scheme~$T$ to the subgroup $\bE(T)$ of the group $\bH(T)$ generated by the subgroups $\bU^+(T)$ and $\bU^-(T)$ of the group $\bH(T)$. (Cf.~\cite[Def.~5.23]{FedorovPanin} and~\cite[Def.~7.2]{FedorovExotic}). As in the proof of~\cite[Prop.~7.1]{FedorovExotic}, we have a diagram
\[
\begin{CD}
\bE(\dot D_Y) @>>>\prod_{i=1}^n\bE(\dot D_{y_i})\\
@VVV @VVV \\
\Gr_\bH(Y) @>>>\prod_{i=1}^n\Gr_\bH(y_i).
\end{CD}
\]
By~\cite[Lm.~7.3]{FedorovExotic} (whose easy proof is valid for any reductive group scheme) the top horizontal map is surjective. Thus it is enough to show that the map
\[
    \bE(\dot D_{y_i})\to\Gr_\bH(y_i)
\]
is surjective for each $i$. Set $k:=k(y_i)$ and $H:=\bH_{y_i}$. Consider an element of $\Gr_\bH(y_i)=\Gr_H(k)$, represented by a pair $(\cE,\tau)$, where $\cE$ is an $H$-bundle over $\P^1_k$, $\tau$ is a trivialization of $\cE$ over $\A^1_k$. By~\cite[Thm.~3.8(a)]{GilleTorseurs}, $\cE$ is Zariski locally trivial. Let us trivialize $\cE$ in a formal neighbourhood of $\infty$, this trivialization and $\tau$ differ by an element $\beta\in H\bigl(k((t))\bigr)$. By construction, the image of $\beta$ under the projection $H\bigr(k((t))\bigl)\to\Gr_\bH(y_i)$ is $(\cE,\tau)$.

Next, $H$ is simple and simply-connected and the field $k((t))$ is infinite. Thus we may use~\cite[Lemma~4.5(1)]{Gille:BourbakiTalk} and~\cite[Fait~4.3(2)]{Gille:BourbakiTalk} to conclude that we can write $\beta=\beta'\beta''$ with $\beta'\in\bE\bigl(k((t))\bigr)=\bE(\dot D_{y_i})$, $\beta''\in H\bigr(k[[t]]\bigl)$. Clearly, $\beta'$ lifts $(\cE,\tau)$ and we are done.
\end{proof}

Note that, instead of using~\cite[Thm.~3.8(a)]{GilleTorseurs} in the proof above, one can use the Grothendieck--Serre conjecture for discrete valuation rings~\cite{Nisnevich1}. The same applies to the reference in the proof of Theorem~\ref{th:FP}.

\subsection{Lifting modifications to the simply-connected central cover}
Let, as before, $\phi\colon G^{\,\sc}\to G$ be the simply-connected central cover of a semisimple $k$-group scheme $G$, where $k$ is a field. This gives a morphism of ind-schemes $\Gr_{G^{\,\sc}}\to\Gr_G$. The goal of this section is to prove the following proposition (cf.~Lemma~\ref{lm:ChangeOfTriv}).
\begin{proposition}\label{pr:liftsTOscGr}
Let $K$ be any field containing $k$. The image of the set $\Gr_{G^{\,\sc}}(K)$ in $\Gr_G(K)$ is $L^+G(K)$-invariant.
\end{proposition}
\begin{proof}
Since $L^+G(K)=L^+G_K(K)$ and $\Gr_G(K)=\Gr_{G_K}(K)$, performing a base change we may assume that $K=k$.

For split group schemes there is a well-known stratification of Grassmannians by $L^+G$-orbits; the orbits are parameterized by the Weyl group orbits in the co-character lattice. If the group scheme is not an inner form, we have a coarser stratification constructed in~\cite{FedorovExotic}. We will recall this stratification now.

Let $G^{\,\spl}$ be the split semisimple $k$-group scheme of the same type as $G$. Let $T^{\,\spl}\subset G^{\,\spl}$ be a maximal (split) torus. Following~\cite[Sect.~5.4.2]{FedorovExotic} put
\begin{equation*}
    X_*:=\Hom(\Gm{k},T^{\,\spl})\subset T^{\,\spl}\bigl(k((t))\bigr).
\end{equation*}
For $\lambda\in X_*$ denote by $t^\lambda$ the corresponding element of $T^{\,\spl}\bigl(k((t))\bigr)$. Abusing notation, we also denote by $t^\lambda$ the projection to $\Gr_{G^{\,\spl}}(k)$ of
\[
    t^\lambda\in T^{\,\spl}\bigl(k((t))\bigr)\subset G^{\,\spl}\bigl(k((t))\bigr).
\]
Denote by $\Gr_{G^{\,\spl}}^\lambda$ the $L^+G^{\,\spl}$-orbit of $t^\lambda$; this is a locally closed subscheme of $\Gr_{G^{\,\spl}}$. We have $\Gr_{G^{\,\spl}}^\lambda=\Gr_{G^{\,\spl}}^\mu$ if and only if~$\lambda$ and $\mu$ are in the same $W$-orbit (here $W$ is the Weyl group of $G^{\,\spl}$). By~\cite[Prop.~5.7]{FedorovExotic}, we get a stratification (in the sense of~\cite[Sect.~5.3]{FedorovExotic})
\begin{equation}\label{eq:stratspl}
    \Gr_{G^{\,\spl}}=\bigcup_{\lambda\in X_*/W}\Gr_{G^{\,\spl}}^\lambda.
\end{equation}
Next, $G$ is a twist of $G^{\,\spl}$ by an $\Aut(G^{\,\spl})$-bundle $\cT$ over $\spec k$, so by~\cite[Prop.~5.4]{FedorovExotic} we get $\Gr_G=\cT\times^{\Aut(G^{\,\spl})}\Gr_{G^{\,\spl}}$. Unfortunately, the orbits $\Gr_{G^{\,\spl}}^\lambda$ are not $\Aut(G^{\,\spl})$-invariant, so we need a coarser stratification. Note that $\Out:=\Aut(G^{\,\spl})/G^{\,\spl,\ad}$ acts on $W$ so we get a semi-direct product $W\leftthreetimes\Out$. For $\hat\lambda\in X_*/(W\leftthreetimes\Out)$, write $\Orb(\hat\lambda)$ for the corresponding $\Out$-orbit on $X_*/W$ and put
\[
    \Gr^{\hat\lambda}_{G^{\,\spl}}:=\bigsqcup_{\lambda\in\Orb(\hat\lambda)}\Gr^\lambda_{G^{\,\spl}}.
\]
Note that, if $\lambda_1,\lambda_2\in\Orb(\hat\lambda)$, then $\Gr^{\lambda_1}_G$ is isomorphic to $\Gr^{\lambda_2}_G$, so these orbits have the same dimension. It follows that $\Gr^{\lambda_1}_G$ cannot lie in the closure of $\Gr^{\lambda_2}_G$. Thus, the above is, in fact, a disjoint union of schemes.

The locally closed subsets $\Gr^{\hat\lambda}_{G^{\,\spl}}$ are $\Aut(G^{\,\spl})$-invariant so we put
\[
    \Gr^{\hat\lambda}_G:=\cT\times^{\Aut(G^{\,\spl})}\Gr^{\hat\lambda}_{G^{\,\spl}}.
\]
Now the stratification~\eqref{eq:stratspl} gives rise to a stratification (\cite[Prop.~5.12]{FedorovExotic})
\begin{equation}\label{eq:strat}
    \Gr_G=\bigcup_{\hat\lambda\in X_*/(W\leftthreetimes\Out)}\Gr_G^{\hat\lambda}.
\end{equation}
Let $G^{\,\sc,\spl}$ be the simply-connected central cover of $G^{\,\spl}$, $T^{\,\sc,\spl}$ be the preimage of $T^{\,\spl}$ in $G^{\,\sc,\spl}$ (this is a maximal split torus in $G^{\,\sc,\spl}$), let $X_*^{\sc}$ be the co-character lattice of $T^{\,\sc,\spl}$. Then we have similarly to the above
\begin{equation*}
    \Gr_{G^{\,\sc}}=\bigcup_{\hat\lambda\in X_*^{\sc}/(W\leftthreetimes\Out)}\Gr_{G^{\,\sc}}^{\hat\lambda};
\end{equation*}
this decomposition is compatible with~\eqref{eq:strat} and the projection $\pi\colon\Gr_{G^{\,\sc}}\to\Gr_G$.

Now we return to the proof of Proposition~\ref{pr:liftsTOscGr}. Consider a point $\alpha\in\Gr_G(k)$. By~\eqref{eq:strat} it belongs to $\Gr_G^{\hat\lambda}(k)$ for some $\hat\lambda\in X_*/(W\leftthreetimes\Out)$. We claim that $\alpha$ lifts to a point of $\Gr_{G^{\,\sc}}(k)$ if and only if $\hat\lambda\in X^{\sc}_*/(W\leftthreetimes\Out)$ (we identify $X_*^{\sc}$ with a sublattice of $X_*$). The proposition follows from this statement because $\Gr_G^{\hat\lambda}$ is manifestly $L^+G$-invariant.

Recall that the projection $\pi\colon\Gr_{G^{\,\sc}}\to\Gr_G$ takes $\Gr_{G^{\,\sc}}^{\hat\lambda}$ to $\Gr_G^{\hat\lambda}$. This proves the `only if' part of our claim. For the converse, it suffices to prove the following lemma.
\begin{lemma}
Assume that $\hat\lambda\in X^{\sc}_*/(W\leftthreetimes\Out)$. Then $\pi$ induces an isomorphism of schemes $\Gr_{G^{\,\sc}}^{\hat\lambda}\to\Gr_G^{\hat\lambda}$.
\end{lemma}
\begin{proof}
First of all, it is enough to prove the statement after passing to an algebraic closure of $k$, in which case $G$ is split and we have a finer stratification~\eqref{eq:stratspl}. Thus we assume that $k$ is algebraically closed and show that for $\lambda\in X_*^{\sc}/W$ the canonical morphism $\pi'\colon\Gr_{G^{\,\sc}}^\lambda\to\Gr_G^\lambda$ is an isomorphism.

We say that a parabolic subgroup scheme $P\subset G$ is of type $\lambda$ if the Weyl group of a Levi factor of $P$ is the stabilizer of $\lambda$ in $W$. Let $F_G^\lambda$ be the scheme of parabolic subgroups of type $\lambda$. In~\cite[Sect.~5.4.3]{FedorovExotic} we have constructed a morphism $\Gr_G^\lambda\to F_G^\lambda$. We have a similar morphism for $G^{\,\sc}$ and a commutative diagram
\begin{equation}\label{eq:CD2}
\begin{CD}
\Gr_{G^{\,\sc}}^\lambda @>\pi'>>\Gr_G^\lambda\\
@VVV @VVV\\
F_{G^{\,\sc}}^\lambda @>>> F_G^\lambda.
\end{CD}
\end{equation}
Note that the lower horizontal morphism is an isomorphism (the proof is analogous to~\cite[Exercise~5.5.8]{ConradReductive}). Since the left projection in the diagram is $G^{\,\sc}$-equivariant and $G^{\,\sc}$ acts transitively on $F_{G^{\,\sc}}^\lambda$, the generic flatness implies that this projection if flat. Similarly, the right projection is flat. Thus it is enough to check that $\pi'$ induces isomorphism of fibers.

Fix a lift of $\lambda$ to $X_*^{\sc}$ so that we have a point $t^\lambda\in\Gr_{G^{\,\sc}}^\lambda(k)$ and a point $t^\lambda\in\Gr_G^\lambda(k)$. Let $C^{\,\sc}$ be the fiber of the morphism $\Gr_{G^{\,\sc}}^\lambda\to F_{G^{\,\sc}}^\lambda$ containing $t^\lambda$; let $C^{\,\sc}$ be the fiber of the morphism $\Gr_G^\lambda\to F_G^\lambda$ containing $t^\lambda$. It is enough to show that $\pi'$ induces an isomorphism $C^{\,\sc}\to C$ because diagram~\eqref{eq:CD2} is $G^{\,\sc}$-equivariant.

For a $k$-group scheme $H$, we denote by $H^{(1)}$ the kernel of the evaluation map $L^+H\to H$. We note that this is just the group scheme of jets into $H$ based at the identity. We claim that $C$ is the $G^{(1)}$-orbit of $t^\lambda$. Indeed, we have a semidirect product decomposition $L^+G=G^{(1)}\leftthreetimes G$. As explained in~\cite[Sect.~5.4.3]{FedorovExotic}, the morphism $\Gr_G^\lambda\to F_G^\lambda$ is induced by the evaluation map
\[
    L^+G=G^{(1)}\leftthreetimes G\to G\to G\cdot t^\lambda=F_G^\lambda.
\]
Let $P^\lambda$ be the stabilizer of $t^\lambda$ in $G$. We see that $C=G^{(1)}P^\lambda\cdot t^\lambda=G^{(1)}\cdot t^\lambda$.

Next, let $U$ be a unipotent subgroup scheme of $G$ opposite to $P^\lambda$. We claim that $C=U^{(1)}\cdot t^\lambda$. To this end, it is enough to check that $G^{(1)}=U^{(1)}\cdot P^{(1)}$ and that $P^{(1)}$ stabilizes $t^\lambda$. For the first statement we note that the multiplication map $U\times P\to G$ induces an isomorphism on the Level of Lie algebras, so it induces an isomorphism of jets based at the identity. Similarly, the second statement reduced to a statement about the Lie algebras.

Let $U^{\,\sc}\subset G^{\,\sc}$ be the preimage of $U$ under the projection $G^{\,\sc}\to G$. Then $U^{\,\sc}$ is a unipotent subgroup scheme opposite to the stabilizer of $t^\lambda$ in $G^{\,\sc}$. Similarly to the above we check that $C^{\,\sc}=(U^{\,\sc})^{(1)}\cdot t^\lambda$.

Note that the central isogeny $\phi\colon G^{\,\sc}\to G$ induces an isomorphism $U^{\,\sc}\to U$. Thus we have an isomorphism $(U^{\,\sc})^{(1)}\to U^{(1)}$. The stabilizer of $t^\lambda$ in $U^{(1)}$ is
\[
    U^{(1)}\cap(t^\lambda\cdot L^+G\cdot t^{-\lambda})=U^{(1)}\cap(t^\lambda\cdot L^+U\cdot t^{-\lambda})
\]
and we have a similar formula for the stabilizer in $(U^{\,\sc})^{(1)}$. Thus the above isomorphism identifies stabilizers, so it induces an isomorphism $C^{\sc}\to C$. The lemma follows.
\end{proof}

The lemma completes the proof of the claim. Proposition~\ref{pr:liftsTOscGr} is proved.
\end{proof}

\begin{remark}
If the characteristic of $k$ does not divide the order of $\pi_1(G)$, it is known that $\pi\colon\Gr_{G^{\,\sc}}\to\Gr_G$ induces an isomorphism between $\Gr_{G^{\,\sc}}$ and the neutral connected component of $\Gr_G$. On the other hand, it is not difficult to derive from the above proof that in general $\pi$ is a morphism from $\Gr_{G^{\,\sc}}$ to the neutral connected component of $\Gr_G$ inducing an isomorphism on $K$-points for every field $K$. The above proposition follows from this fact because the neutral component is preserved under the action of the connected group scheme $L^+G$. One expects that this morphism is a universal homeomorphism. We refer the reader to~\cite[Prop.~3.5]{HainesRicharzNormality} for a similar statement.
\end{remark}

\subsection{Principal bundles with topologically trivial fibers over families of affine lines}
In this section we prove an analogue of~\cite[Thm.~3]{FedorovPanin} and of~\cite[Thm.~1.8]{PaninFiniteFieldsIzvestiya} where the group scheme is allowed to be arbitrary reductive but the $\bG$-bundle is required to be topologically trivial on closed fibers. Recall that a~semisimple group scheme over a scheme $U$ is called \emph{isotropic} if it contains a one-dimensional torus $\Gm{U}$. If $U$ is connected affine, and semi-local, then by~\cite[Exp.~XXVI, Cor.~6.14]{SGA3-3} this is equivalent to the group scheme containing a proper parabolic subgroup scheme. For any scheme $S$ we denote by $\Pic(S)$ the group of isomorphism classes of line bundles over $S$. Recall that $\bZ$ is the center of $\bG$ and $\bG^{\ad}=\bG/\bZ$.

\begin{theorem}\label{th:MainThm2}
Let $U$ be a connected affine semi-local scheme over a field. Let $\bG$ be a~reductive group scheme over $U$. Let $\bG^1$, \ldots, $\bG^r$ be the factors of the adjoint group scheme $\bG^{\ad}$ \parl{}cf.~equation~\eqref{eq:prod}\parr. Let $Z\subset\A^1_U$ be a closed subscheme finite over $U$. Let $Y\subset\A^1_U$ be a closed subscheme finite and \'etale over $U$. Assume that $Y\cap Z=\emptyset$. Assume further that for each $i=1,\ldots,r$ there is a connected component $Y^i\subset Y$ satisfying two properties: {\rm(}i{\rm)} $(\bG^i)_{Y^i}$ is isotropic and {\rm(}ii{\rm)} for every closed point $u\in U$ such that $\bG^i_u$ is isotropic we have $\Pic(\P^1_u-Y^i_u)=0$. Finally, assume that
the relative line bundle $\cO_{\P^1_U}(1)$ trivializes on $\P^1_U-Y$.

Let $\cG$ be a~principal $\bG$-bundle over $\P^1_U$ such that its restriction to $\P^1_U-Z$ is trivial and such that for all closed points $u\in U$ the $\bG_u^{\ad}$-bundle $(\cG|_{\P^1_u})/\bZ_u$ is topologically trivial. Then the restriction of $\cG$ to $\P^1_U-Y$ is also trivial.
\end{theorem}
\begin{remarks}
{\rm(}i{\rm)} We are not assuming that the components $Y^i$ are different for different $i$.\\
{\rm(}ii{\rm)} The condition that $\cO_{\P^1_U}(1)$  trivializes on $\P^1_U-Y$ is necessary. Indeed, if we take $\bG=\Gm{U}$, $\cG=\cO_{\P_U^1}(1)^\times$, $Y=\emptyset$, then $\cG$ is not trivial over $\P^1_U-Y$. (Note that $Y$ satisfies the other conditions of the theorem because $r=0$.)\\
{\rm(}iii{\rm)} Note that the $\bG^{\ad}_u$-bundle $(\cG|_{\P^1_u})/\bZ_u$ is Zariski locally trivial, because $\cG|_{\P^1_u}$ is trivial over $\P^1_u-Z_u$.\\
{\rm(}iv{\rm)} Assume that the residue fields of the closed points of $U$ are infinite. Then we may start with $Y,Z\subset\P^1_U$. Indeed, applying a projective transformation of $\P^1_U$ we can always achieve $Y,Z\subset\A^1_U$. The condition $Y\cap Z=\emptyset$ is also not necessary in this case, see Remark~2 after~\cite[Thm.~3]{FedorovPanin}.\\
{\rm(}v{\rm)} The proof of this theorem is much simpler in many cases: for example, if $U$ is local or normal. When $U$ is not normal, the problem is that a line bundle on $\P^1_U-Y$ need not be trivial, unless it can be extended to $\P^1_U$.
\end{remarks}

We need a proposition, which is a slight generalization of~\cite[Prop.~9.6]{PaninStavrovaVavilov}.
\begin{proposition}\label{pr:Horrocks}
Let, as above, $U$ be a connected affine semi-local scheme over a field. Let $\bH$ be a semisimple $U$-group scheme. Let $\cH$ be an $\bH$-bundle over $\P^1_U$ such that for every closed point $u\in U$ the restriction of $\cH$ to $\P^1_u$ is a trivial $\bH_u$-bundle. Then~$\cH$ is isomorphic to the pullback of an $\bH$-bundle over $U$.
\end{proposition}
\begin{proof}
Since $\bH$ is semisimple, there is an embedding $\bH\hookrightarrow\GL_{n,U}$ for some $n$ by~\cite[Cor.~3.2]{ThomasonResolution}. The rest of the proof is completely analogous to that of~\cite[Prop.~9.6]{PaninStavrovaVavilov}.
\end{proof}

\begin{proof}[Proof of Theorem~\ref{th:MainThm2}] \emph{Step 1}. Let $\scG^i$ be the simply-connected central cover of the group scheme $\bG^i$ (see~\cite[Exercise~6.5.2]{ConradReductive}). Then $\prod_{i=1}^r\scG^i$ is the simply-connected central cover of $\bG^{\ad}$. \emph{We claim that the covering homomorphism $\prod_{i=1}^r\scG^i\to\bG^{\ad}$ lifts to a homomorphism $\prod_{i=1}^r\scG^i\to\bG$.} Indeed, let $[\bG,\bG]$ be the derived subgroup scheme of $\bG$, then the morphism $[\bG,\bG]\to\bG^{\ad}$ is a central isogeny, so the simply-connected central cover of $[\bG,\bG]$ is also the simply-connected central cover of $\bG^{\ad}$. Hence, $\prod_{i=1}^r\scG^i\to\bG^{\ad}$ factors through $[\bG,\bG]$ and the statement follows. Thus we have a sequence of homomorphisms
\[
    \prod_{i=1}^r\scG^i\to\bG\to\bG^{\ad}=\prod_{i=1}^r\bG^i.
\]

\emph{Step 2}. Let $u\in U$ be a closed point and put $\cG_u:=\cG|_{\P^1_u}$. By assumption (cf.~Definition~\ref{def:TopTriv}), the $\bG^{\ad}_u$-bundle $\cG_u/\bZ_u$ lifts to a Zariski locally trivial $\prod_{i=1}^r\scG_u^i$-bundle $\tilde\cG_u$ over $\P^1_u$. This corresponds to a sequence $(\tilde\cG_u^1,\ldots,\tilde\cG_u^r)$, where $\tilde\cG_u^i$ is a $\scG^i_u$-bundle. Let $\cG_u^i$ be the pushforward of $\tilde\cG_u^i$ to $\bG_u^i$. \emph{We claim that $\tilde\cG_u^i$ is trivial over $\P^1_u-Y_u^i$ for all $i$.} Indeed, if $\scG_u^i$ is anisotropic, this follows immediately from~\cite[Thm.~3.10(a)]{GilleTorseurs}. If $\scG_u^i$ is isotropic, then $\bG_u^i$ is also isotropic (see~\cite[Thm.~2.20]{Borel-Tits2}) so $\Pic(\P^1_u-Y_u^i)=0$, and the statement again follows from~\cite[Thm.~3.10(a)]{GilleTorseurs}.

For $i=1,\ldots,r$ choose a trivialization $\tilde\tau_u^i$ of $\tilde\cG_u^i$ over $\P^1_u-Y_u^i$. These trivializations induce trivializations of each $\tilde\cG_u^i$ over $\P^1_u-Y_u$, which, in turn, give a trivialization of $\cG_u^i$ on $\P^1_u-Y_u$. Denote this trivialization by $\tau_u^i$.

\emph{Step 3}. Let $\cF_u^i$ be the trivial $\bG_u^i$-bundle over $\P^1_u$. Then $(\cF_u^i,\tau_u^i)$ is a modification of $\cG_u^i$ at $Y_u$. Choose a trivialization of $\cG$ over $\P^1_U-Z$. Since $Y\cap Z=\emptyset$, this gives a~trivialization of $\cG_u$ (and, in turn, of $\cG_u^i$) in a neighborhood of $Y_u\subset\P^1_u$.
The latter trivialization allows us to identify modifications with sections of the affine Grassmannian, so that $(\cF_u^i,\tau_u^i)$ corresponds to $\alpha_u^i\in\Gr_{\bG^i}(Y_u)$.

\begin{lemma}
    {\rm(}i{\rm)} $\alpha_u^i|_{Y_u-Y_u^i}=\Id_{\Gr}$.\\
    {\rm(}ii{\rm)} $\alpha_u^i|_{Y_u^i}$ can be lifted to $\tilde\alpha_u^i\in\Gr_{\scG^i}(Y_u^i)$.
\end{lemma}

\begin{proof}
Consider any trivialization $\tilde\sigma_u^i$ of $\tilde\cG_u^i$ in a Zariski neighborhood of $Y_u$. This induces a trivialization $\sigma_u^i$ of $\cG_u^i$ in the same neighborhood. These trivializations allow us to identify modifications with sections of affine Grassmannians. In particular, denoting by $\tilde\cF_u^i$ the trivial $\scG^i_u$-bundle over $\P^1_u$, we get a modification $(\tilde\cF_u^i,\tilde\tau_u^i)$ of $\tilde\cG_u^i$ and thus a section $\tilde\beta^i_u\in\Gr_{\scG^i}(Y_u^i)$. We extend $\tilde\beta^i_u$ to an element of $\Gr_{\scG^i}(Y_u)$ (which we also denote by $\tilde\beta^i_u$) by setting $\tilde\beta_u^i|_{Y_u-Y^i_u}=\Id_{\Gr}$. Equivalently, this extension corresponds to the modification $(\tilde\cF_u^i,\tilde\tau_u^i|_{\P^1_u-Y_u})$.

Let $\beta_u^i$ be the image of $\tilde\beta_u^i$ under the projection $\Gr_{\scG^i}(Y_u)\to\Gr_{\bG^i}(Y_u)$. It follows from the construction that $\alpha_u^i$ and $\beta_u^i$ correspond to the same modification of the same $\bG_u^i$-bundle but with respect to different trivialization of this bundle near $Y_u$. According to Lemma~\ref{lm:ChangeOfTriv}, $\alpha_u^i$ differs from $\beta_u^i$ by an action of an element of $L^+\bG^i(Y_u)$. Since $\Id_{\Gr}\in\Gr_{\bG^i}(Y_u-Y_u^i)$ is $L^+\bG^i(Y_u-Y_u^i)$-invariant, we obtain statement~(i). Statement~(ii) follows from Proposition~\ref{pr:liftsTOscGr}, applied to each point of $Y_u^i$, and the fact that $\beta_u^i$ lifts to $\tilde\beta_u^i$.
\end{proof}

\emph{Step 4}. Let $U^{\,\mathrm{cl}}$ be the set of closed points of $U$. Let $\tilde\alpha_u^i$ be as in the above lemma. Recall that $\bG^i$ is the Weil restriction of a simple $U'$-group scheme (call it $\bG'$) via a finite \'etale morphism $U'\to U$ with connected $U'$. It is easy to see that $\scG^i$ is the Weil restriction of the simply-connected cover of $\bG'$ via the same morphism. Also, $(\scG^i)_{Y^i}$ contains a proper parabolic subgroup scheme because $(\bG^i)_{Y^i}$ does (cf.~\cite[Exercise~5.5.8]{ConradReductive}). Thus, by Lemma~\ref{lm:ExtendSect} the collection  $(\tilde\alpha_u^i|u\in U^{\,\mathrm{cl}})$ lifts to a point $\tilde\alpha^i\in\Gr_{\scG^i}(Y^i)$. We extend this to a point of $\Gr_{\scG^i}(Y)$ by setting $\tilde\alpha^i|_{Y-Y^i}=\Id_{\Gr}$. The collection $(\tilde\alpha^i|i=1,\ldots,r)$ gives rise to a section $\alpha\in\Gr_\bG(Y)$. Since we have trivialized $\cG$ in a neighborhood of $Y$, this gives a modification $(\cF,\tau)$ of $\cG$ at $Y$. By construction the $\bG^{\ad}_u$-bundle $(\cF|_{\P^1_u})/\bZ_u$ is trivial for $u\in U^{\,\mathrm{cl}}$. Now, by Proposition~\ref{pr:Horrocks}, the $\bG^{\ad}$-bundle $\cF/\bZ$ is isomorphic to the pullback of a $\bG^{\ad}$-bundle under the projection $\P^1_U\to U$. On the other hand, since $\cF$ is a modification of $\cG$ at $Y$, the $\bG^{\ad}$-bundle $(\cF/\bZ)|_{U\times\infty}\simeq(\cG/\bZ)|_{U\times\infty}$ is trivial. It follows that $\cF/\bZ$ is trivial. Now, it follows from the exact sequence for non-abelian cohomology groups, that there is a $\bZ$-bundle $\cZ$ over $\P^1_U$ such that $\cF$ is isomorphic to the pushforward of $\cZ$.

\emph{Step 5}. Note that the center of a reductive group scheme is a group scheme of multiplicative type. Recall that the relative line bundle $\cO_{\P^1_U}(1)$ trivializes on $\P_U^1-Y$. The following lemma is somewhat similar to~\cite[Lemma~2.4]{ColliotTheleneSansuc}.
\begin{lemma}\label{lm:Z-bundles}
Let $U$ and $Y$ be as before; let $\bZ$ be a group scheme of multiplicative type over $U$. Let $\cZ$ be a $\bZ$-bundle over $\P^1_U$. Then $\cZ|_{\P^1_U-Y}$ is isomorphic to the pullback of a $\bZ$-bundle over $U$.
\end{lemma}
\begin{proof}
Since $\bZ$ is not smooth in general, we will work in the fppf topology over~$U$. We claim that there is a unique co-character $\lambda\colon\Gm{U}\to\bZ$ such that $\cZ':=\lambda_*\cO_{\P^1_U}(1)^\times$ and $\cZ$ are isomorphic locally in the fppf topology over $U$. Indeed, the statement is local over $U$, so we may assume that $\bZ$ is split. Then the question reduces to the cases $\bZ=\Gm{U}$ and $\bZ=\mu_{n,U}$, where $\mu_{n,U}$ is the group scheme of $n$-th roots of unity. The first case is a statement about line bundles; we leave it to the reader. The second case reduces to the statement that a $\mu_{n,U}$-bundle over $\P^1_U$ is trivial fppf locally over the base, which follows easily from the exact sequence $1\to\mu_{n,U}\to\Gm{U}\to\Gm{U}\to1$; the claim is proved.

We see that $\cZ\simeq\cZ'\otimes p^*\cZ''$, where $p\colon\P^1_U\to U$ is the projection, $\cZ''$ is a $\bZ$-bundle over $U$ (note that $\bZ$ is a commutative group scheme so the tensor product of $\bZ$-bundles makes sense). It remains to notice that $\cZ'=\lambda_*\cO_{\P^1_U}(1)^\times$ is trivial on $\P_U^1-Y$ because $\cO_{\P^1_U-Y}(1)$ is trivial. Lemma~\ref{lm:Z-bundles} is proved.
\end{proof}

We see that $\cF|_{\P^1_U-Y}$ is isomorphic to the pullback of a $\bG$-bundle over $U$. Since~$\cF$ and $\cG$ are isomorphic over $U\times\infty$ and $\cG$ is trivial over $U\times\infty$, we see that $\cF|_{\P^1_U-Y}$ is trivial. Finally, $\cG$ and $\cF$ are isomorphic over $\P^1_U-Y$, and Theorem~\ref{th:MainThm2} is proved.
\end{proof}

\subsection{Proof of  Theorem~\ref{th:FP}}\label{sect:ProofFP} We use the notation from the formulation of the theorem. We may assume that $U$ is connected. Applying an affine transformation to $\A^1_U$, we may assume that $\Delta$ is the horizontal section $\Delta(U)=U\times1$. We can extend the $\bG$-bundle $\cE$ to a $\bG$-bundle $\tilde\cE$ over $\P^1_U$ by gluing it with the trivial $\bG$-bundle over $\P^1_U-Z$. Let $\phi\colon\bG^{\sc}\to\bG^{\ad}$ be the simply-connected central cover (see~\cite[Exercise~6.5.2]{ConradReductive}); let $d$ be the degree of $\phi$. Consider the morphism $\P^1_\Z\to\P^1_\Z\colon z\mapsto z^d$; let $\psi\colon\P^1_U\to\P^1_U$ be the base change of this morphism. Consider the $\bG$-bundle $\psi^*\tilde\cE$ over $\P^1_U$. For a closed point $u\in U$ write $\tilde\cE_u:=\tilde\cE|_{\P_u^1}$. Then by~\cite[Thm.~3.8(a)]{GilleTorseurs} the $\bG^{\ad}_u$-bundle $\tilde\cE_u/\bZ_u$ is Zariski locally trivial. By Proposition~\ref{pr:PullbackTopTriv} the $\bG^{\ad}_u$-bundle $\psi^*\tilde\cE_u/\bZ_u$ is topologically trivial. Since the morphism $\psi$ has a section over $U\times1$, it is enough to show that $\psi^*\tilde\cE|_{U\times1}$ is trivial.

\emph{Case 1. $U$ is a scheme over an infinite field $k$.} By~\cite[Prop.~4.1]{FedorovPanin} for $i=1,\ldots,r$, we can find a closed subscheme $Y^i\subset\P^1_U$ finite and \'etale over $U$ such that $(\bG^i)_{Y^i}$ is isotopic and for every closed point $u\in U$ such that $\bG_u^i$ is isotropic we have a $k(u)$-rational point on the fiber $Y^i_u$ (which implies $\Pic(\P^1_u-Y^i_u)=0$). Since $k$ is infinite and $U$ is semi-local, we can shift the subschemes $Y^i$ so that they do not intersect each other, $\psi^{-1}(Z)$, and $U\times 1$. Again, since $k$ is infinite, we have $a\in k$ such that $U\times a$ does not intersect $\psi^{-1}(Z)$. It remains to take $Y=\sqcup_{i=1}^rY^i\sqcup(U\times a)$ and apply Theorem~\ref{th:MainThm2} to $\psi^*\tilde\cE$.

\emph{Case 2. The residue fields of points of $U$ are finite over $k$.} By~\cite[Lemma~2.1]{PaninFiniteFieldsIzvestiya} applied to the identity morphism $U\to U$, we can find field extensions $k'\supset k$ and $k''\supset k$ of coprime degrees and a closed $U$-embedding
\[
    (U\times_k\spec k')\bigsqcup(U\times_k\spec k'')\hookrightarrow\A^1_U
\]
such that the image $Y_1$ of this embedding does not intersect $\psi^{-1}(Z)$. Note that the relative line bundle $\cO(1)$ trivializes on $\P^1_U-Y_1$. By~\cite[Prop.~3.2]{PaninFiniteFieldsIzvestiya} we can find an \'etale and finite over $U$ subscheme $Y_2\subset\A^1_U$ such that $\bG^{\ad}_{Y_2}$ is quasisplit, for all points $u\in U^{\,\mathrm{cl}}$ we have $\Pic(\P^1_u-(Y_2)_u)=0$, and
\[
    Y_2\cap(\psi^{-1}(Z)\cup Y_1)=\emptyset.
\]
(Note that the proposition is only formulated for simple simply-connected group schemes but the proof goes through for all semisimple group schemes. The proposition also requires that $Z$ is \'etale over $U$ but this is also not used in the proof.) It is easy to see that the factors of a quasisplit semisimple group scheme are quasisplit and thus isotropic by~\cite[Exp.~XXVI, Cor.~6.14]{SGA3-3}. Then $Y=Y_1\cup Y_2$ satisfies the conditions of Theorem~\ref{th:MainThm2} (one takes $Y_i=Y_2$ for all~$i$). It remains to apply Theorem~\ref{th:MainThm2} to $Y$ and $\psi^*\tilde\cE$. \qed

\subsection{Proof of  Theorem~\ref{th:FP2}}\label{sect:ProofFP2}
We use the notation from the formulation of the theorem. As in the proof of Theorem~\ref{th:FP}, we extend the $\bG$-bundle $\cE$ to a $\bG$-bundle $\tilde\cE$ over $\P^1_U$ and assume that $\Delta(U)=U\times1$. Let $\psi$ and $\tilde\cE_u$ be as in the proof of Theorem~\ref{th:FP}, then  $\psi^*\tilde\cE_u/\bZ_u$ is topologically trivial for every closed point $u\in U$. It is enough to show that~$\psi^*\tilde\cE|_{\P_U^1-(U\times0)}$ is trivial.  By assumption, we can embed $\bG$ into $\GL_{n,U}$. By~\cite[Korollar~3.5.2]{MoserGrSerre} we may assume that~$U$ is local (note that $\P^1_U-(U\times 0)\simeq\A^1_U$).

In the same way as in the proof of Theorem~\ref{th:FP} we find a closed subscheme $Y\subset\P^1_U$ finite and \'etale over $U$ such that $\psi^*\tilde\cE$ is trivial over $\P^1_U-Y$. Note that such $Y$ may be chosen so that it does not intersect any given closed subscheme of $\A^1_U$ as long as this subscheme is finite over $U$. In particular, we may assume that $Y\cap(U\times0)=\emptyset$. Since $\bG$ is strongly locally isotropic and $U$ is local, each $\bG^i$ is locally isotropic. Thus, we can apply Theorem~\ref{th:MainThm2} taking $Y$ for $Z$ and $U\times0$ for $Y$. We see that $\psi^*\tilde\cE$ is trivial over $\P^1_U-(U\times0)$, which completes the proof of the theorem. \qed

\section{Proofs of Theorem~\ref{th:main} and of Theorem~\ref{th:GrSerre}}
In this Section we derive Theorem~\ref{th:GrSerre} and Theorem~\ref{th:main} from Theorem~\ref{th:FP} and Theorem~\ref{th:FP2} respectively. The proofs are based on~\cite[Thm.~1.5]{PaninNiceTriples}. Note that these derivations are similar to those given in~\cite{FedorovPanin,PaninFiniteFieldsIzvestiya,PaninStavrovaVavilov}; we present them here for the sake of completeness.

\subsection{Proof of Theorem~\ref{th:GrSerre}}\label{sect:ProofOfGrSerre}
\emph{Step 1.} We may assume that $U$ is the semi-local scheme of finitely many closed points $x_1,\ldots,x_n$ on a smooth irreducible $k$-variety~$X$, where~$k$ is a field. Indeed, let $U=\spec R$ and let $k$ be the prime field of $R$ (or any other perfect field contained in $R$). Then, by Popescu's Theorem~\cite{Popescu,SwanOnPopescu,SpivakovskyPopescu}, we can write $U=\lim\limits_{\longleftarrow}U_\alpha$, where $U_\alpha$ are affine schemes smooth and of finite type over $k$. Modifying the system $(U_\alpha)$, we may assume that $U_\alpha$ are integral schemes. A~standard argument shows that there is an index $\alpha$, a reductive group scheme $\bG_\alpha$ over $U_\alpha$ such that $\bG_\alpha|_U=\bG$, and a $\bG_\alpha$-bundle $\cE_\alpha$ over $U_\alpha$ trivial over the generic point of $U_\alpha$ and such that the pullback of $\cE_\alpha$ to $U$ is isomorphic to $\cE$. Let $y_1,\ldots,y_n\in U_\alpha$ be the images of all closed points of $U$. For $i=1,\ldots,n$ choose a closed point $x_i\in U_\alpha$ in the Zariski closure of $y_i$. Let $R'$ be the semi-local ring of $x_1,\ldots,x_n$ on $X:=U_\alpha$. Let $\bG'$ be the restriction of $\bG_\alpha$ to $U':=\spec R'$. The morphism $U\to U_\alpha$ factors through $U'$. Thus it is enough to prove the theorem for $U'$, $\bG'$, and $\cE':=\cE_\alpha\times\times_{U_\alpha}U'$.

\emph{Step 2.} Replacing $X$ by a Zariski neighbourhood of $\{x_1,\ldots,x_n\}$, we may assume that there are a group scheme $\bG_X$ over $X$ such that $\bG_X|_U=\bG$, a $\bG_X$-bundle $\cE'$ over $X$ such that $\cE'|_U=\cE$, and a non-zero function $f\in H^0(X,\cO_X)$ such that the restriction of $\cE'$ to $X_f$ is a trivial bundle.

\emph{Step 3.} We keep the notation from Step~2. Multiplying $f$ by an appropriate function, we may assume that $f$ vanishes at each $x_i$.
Our goal is to construct a $\bG$-bundle $\cG$ over $\A^1_U$ by \'etale descent such that $\cG|_{U\times0}\simeq\cE$. Then we can apply Theorem~\ref{th:FP} to conclude that $\cE$ is trivial. The construction of this $\cE$ is standard and is achieved by using a certain diagram. Precisely, by~\cite[Thm.~1.5]{PaninNiceTriples} there is a monic polynomial $h\in H^0(U,\cO_U)[t]$, a commutative diagram with an irreducible affine $U$-smooth $Y$:
\begin{equation}\label{eq:CD}
\begin{CD}
    (\A^1_U)_h @<<< Y_{\tau^*(h)} @>p_X|_{Y_\tau^*(h)}>>X_f\\
    @V inc VV @V inc VV @V inc VV \\
    \A^1_U @<\tau<< Y @>p_X>>X,
\end{CD}
\end{equation}
and a morphism $\delta\colon U\to Y$ satisfying the following conditions:\\
(i) the left square is an elementary distinguished square in the category of affine $U$-smooth schemes in the sense of~\cite[Sect.~3.1, Def.~1.3]{MorelVoevodsky}; this means that the vertical maps are open embeddings, the horizontal maps are \'etale, and $\tau$ induces an isomorphism
\[
    \tau^{-1}(\{h=0\})_{red}\to\{h=0\}_{red};
\]
(ii) $p_X\circ\delta=can\colon U\to X$, where $can$ is the canonical morphism;\\
(iii) $\tau\circ\delta=i_0\colon U\to\A^1_U$ is the zero section of the projection $pr_U\colon \A^1_U\to U$.\\
(iv) for $p_U:=pr_U\circ\tau$ there is a $Y$-group scheme isomorphism $\Phi\colon p_U^*(\bG)\to p_X^*(\bG_X)$
with $\delta^*(\Phi)=id_{\bG}$.

\emph{Step 4.} We use part~(iv) of Step~3 to view $p_X^*\cE'$ as a $\bG$-bundle. We use the left square from part (i) of Step~3 to glue the trivial $\bG$-bundle over $(\A^1_U)_h$ with $p_X^*\cE'$ to get a $\bG$-bundle $\cG$ over $\A^1_U$. We have
\begin{equation}\label{eq:isoSch}
    \cE=can^*\cE'=\delta^* p_X^*\cE'=\delta^*\tau^*\cG=i_0^*\cG
\end{equation}
so it remains to show that $i_0^*\cG$ is trivial. But $\{h=0\}$ is a closed subscheme of $\A^1_U$ and it is finite over $U$ because $h$ is monic. The residues of all closed points of $U$ are finite extensions of $k$, so they are finite if $k$ is finite. Thus we can apply Theorem~\ref{th:FP} and conclude that $i_0^*\cG$ is trivial. \qed

\begin{remark}
A priori,~\eqref{eq:isoSch} is an isomorphism of $U$-schemes. This is enough for our purposes because a principal bundle is trivial if and only if it has a section, so that triviality does not depend on the group scheme action. On the other hand, using the equation $\delta^*(\Phi)=id_{\bG}$, one can show that~\eqref{eq:isoSch} is compatible with the action of the group scheme, see~\cite[Sect.~6]{PaninNiceTriples}.
\end{remark}

\subsection{Proof of Theorem~\ref{th:main}}\label{sect:ProofOfMain}
\emph{Step 1.} Replacing $k$ by its prime subfield $k'$ and $W$ by $W\times_{k'}k$, we may assume that $k$ is perfect. Next, we may assume that $W$ is of finite type over~$k$. Indeed, write $W=\lim\limits_{\longleftarrow}W_\alpha$, where $W_\alpha$ are $k$-schemes of finite type. Since $\cE$ is affine and finitely presented over $W\times_k U$, there is an index $\alpha$ and a $\bG$-bundle $\cE_\alpha$ over $W_\alpha\times_k U$ such that $\cE$ is isomorphic to the pullback of $\cE_\alpha$ to $W\times_k U$. Next, there is an index $\beta>\alpha$ such that the pullback of $\cE_\alpha$ to $W_\beta\times_k U$ (call it $\cE_\beta$) is trivial over $W_\beta\times_k\Omega$. We see that it is enough to prove the theorem with $W$ and $\cE$ replaced by $W_\beta$ and $\cE_\beta$.

\emph{Step 2.} Similarly to Step~1 of the proof of Theorem~\ref{th:GrSerre}, we may assume that~$U$ is the semi-local scheme of finitely many closed points $x_1,\ldots,x_n$ on a smooth irreducible $k$-variety $X$. In more details, by Popescu's Theorem we can write $U=\lim\limits_{\longleftarrow}U_\alpha$, where $U_\alpha$ are affine schemes smooth and of finite type over $k$. We may assume that $U_\alpha$ are integral schemes. Then we find an index $\alpha$, a reductive strongly locally isotropic group scheme $\bG_\alpha$ over $U_\alpha$ such that $\bG_\alpha|_U=\bG$, and a $\bG_\alpha$-bundle $\cE_\alpha$ over $W\times_kU_\alpha$ trivial over $W\times_k\Omega_\alpha$, where $\Omega_\alpha$ is the generic point of $U_\alpha$ and such that the pullback of $\cE_\alpha$ to $W\times_kU$ is isomorphic to $\cE$. Then it is enough to prove the theorem with $U$ replaced by an appropriate semi-local ring of finitely many closed points of $U_\alpha$.

\emph{Step 3.} Set $U':=W\times_kU$, $X':=W\times_kX$. Similarly to Step~2 of the proof of Theorem~\ref{th:GrSerre}, we may assume that there is a group scheme $\bG_X$ over $X$ such that $\bG_X|_U=\bG$, a $\bG_X$-bundle $\cE'$ over $X'$ such that $\cE'|_{U'}=\cE$, and a non-zero function $f\in H^0(X,\cO_X)$ such that the restriction of $\cE'$ to $X'_f$ is a trivial bundle.

\emph{Step 4.} Similarly to Step~3 of the proof of Theorem~\ref{th:GrSerre}, we find a monic polynomial $h\in H^0(U,\cO_U)[t]$, a commutative diagram~\eqref{eq:CD} with an irreducible affine $U$-smooth~$Y$, and a morphism $\delta\colon U\to Y$ satisfying the same conditions.

\emph{Step 5.} Set $Y':=W\times_kY$. The diagram~\eqref{eq:CD} is a diagram over $k$. Thus we can multiply this diagram by $W$, getting a monic polynomial $h'\in H^0(U',\cO_{U'})[t]$ and a commutative diagram
\[
\begin{CD}
    (\A^1_{U'})_{h'} @<<< Y'_{(\tau')^*(h')} @>p_{X'}|_{Y'_{(\tau')^*(h')}}>>X'_f\\
    @V inc VV @V inc VV @V inc VV \\
    \A^1_{U'} @<\tau'<< Y' @>p_{X'}>>X'.
\end{CD}
\]
We also get a morphism $\delta'\colon U'\to Y'$. These data satisfy the following conditions:\\
(i) the left hand side square is an elementary distinguished square in the category of
affine $U'$-smooth schemes in the sense of~\cite[Sect.~3.1, Def.~1.3]{MorelVoevodsky};\\
(ii) $p_{X'}\circ\delta'=can\colon U'\to X'$, where $can$ is the canonical morphism;\\
(iii) $\tau'\circ\delta'=i'_0\colon U'\to\A^1_{U'}$ is the zero section of the projection $pr_{U'}\colon\A^1_{U'}\to U'$.\\

\emph{Step 6.} We use part~(iv) of Step~4 of the proof of Theorem~\ref{th:GrSerre} to view $p_{X'}^*\cE'$ as a $\bG$-bundle. We use the left square from part (i) of Step~5 to glue the trivial $\bG$-bundle over $(\A^1_{U'})_{h'}$ with $p_{X'}^*\cE'$ to get a $\bG$-bundle $\cG$ over $\A^1_{U'}$. We have
\begin{equation*}
    \cE=can^*\cE'=(\delta')^* p_{X'}^*\cE'=(\delta')^*(\tau')^*\cG=(i'_0)^*\cG,
\end{equation*}
so it remains to show that $(i'_0)^*\cG$ is trivial. But $\bG$ can be embedded into $\GL_{n,U}$ for some $n$ because $U$ is regular and, in particular, normal (see~\cite[Cor.~3.2]{ThomasonResolution}). Thus $\bG_{U'}$ can be embedded into $\GL_{n,U'}$. Next, $\{h'=0\}$ is a closed subscheme of $\A^1_{U'}$ and it is finite over $U'$ because $h'$ is monic. Thus we can apply Theorem~\ref{th:FP2} and conclude that $(i'_0)^*\cG$ is trivial. \qed

\bibliographystyle{../../alphanum}
\bibliography{../../RF}

\begin{thebibliography}{Pan4}

\bibitem[BT]{Borel-Tits2}
Armand Borel and Jacques Tits.
\newblock Compl\'ements \`a l'article: ``{G}roupes r\'eductifs''.
\newblock {\em Inst. Hautes \'Etudes Sci. Publ. Math.}, (41):253--276, 1972.

\bibitem[Con]{ConradReductive}
Brian Conrad.
\newblock Reductive group schemes.\\
\newblock \url{http://math.stanford.edu/~conrad/papers/luminysga3smf.pdf},
  2014.

\bibitem[CTS]{ColliotTheleneSansuc}
Jean-Louis Colliot-Th{\'e}l{\`e}ne and Jean-Jacques Sansuc.
\newblock Principal homogeneous spaces under flasque tori: applications.
\newblock {\em J. Algebra}, 106(1):148--205, 1987.

\bibitem[DG1]{SGA3-2}
Michel Demazure and Alexander Grothendieck.
\newblock {\em Sch\'emas en groupes. {II}: {G}roupes de type multiplicatif, et
  structure des sch\'emas en groupes g\'en\'eraux}.
\newblock S\'eminaire de G\'eom\'etrie Alg\'ebrique du Bois Marie 1962/64 (SGA
  3). Dirig\'e par M. Demazure et A. Grothendieck. Lecture Notes in
  Mathematics, Vol. 152. Springer-Verlag, Berlin, 1970.

\bibitem[DG2]{SGA3-3}
Michel Demazure and Alexander Grothendieck.
\newblock {\em Sch\'emas en groupes. {III}: {S}tructure des sch\'emas en
  groupes r\'eductifs}.
\newblock S\'eminaire de G\'eom\'etrie Alg\'ebrique du Bois Marie 1962/64 (SGA
  3). Dirig\'e par M. Demazure et A. Grothendieck. Lecture Notes in
  Mathematics, Vol. 153. Springer-Verlag, Berlin, 1970.

\bibitem[Fed]{FedorovExotic}
Roman Fedorov.
\newblock Affine {G}rassmannians of group schemes and exotic principal bundles
  over {$\mathbb A^1$}.
\newblock {\em Amer.~Journal of Math.}, 138(4):879--906, 2016.

\bibitem[FP]{FedorovPanin}
Roman Fedorov and Ivan Panin.
\newblock A proof of the {G}rothendieck--{S}erre conjecture on principal
  bundles over regular local rings containing infinite fields.
\newblock {\em Publications math\'ematiques de l'IH\'ES}, 122(1):169--193,
  2015.

\bibitem[Gil1]{GilleTorseurs}
Philippe Gille.
\newblock Torseurs sur la droite affine.
\newblock {\em Transform. Groups}, 7(3):231--245, 2002.

\bibitem[Gil2]{Gille:BourbakiTalk}
Philippe Gille.
\newblock Le probl\`eme de {K}neser-{T}its.
\newblock {\em Ast\'erisque}, (326):Exp. No. 983, vii, 39--81 (2010), 2009.
\newblock S{\'e}minaire Bourbaki. Vol. 2007/2008.

\bibitem[Gro]{GrothendieckFGA}
Alexander Grothendieck.
\newblock Technique de descente et th\'eor\`emes d'existence en g\'eometrie
  alg\'ebrique. {I}. {G}\'en\'eralit\'es. {D}escente par morphismes
  fid\`element plats.
\newblock In {\em S\'eminaire {B}ourbaki, {V}ol.\ 5, {E}xp.\ {N}o.\ 190.},
  pages 299--327. Soc. Math. France, Paris, 1995.

\bibitem[HR]{HainesRicharzNormality}
Thomas~J. Haines and Timo Richarz.
\newblock Normality and cohen-macaulayness of parahoric local models, 2019.

\bibitem[Mos]{MoserGrSerre}
Lukas-Fabian Moser.
\newblock Rational triviale {T}orseure und die {S}erre--{G}rothendiecksche
  {V}ermutung.
\newblock {\em Diplomarbeit, Mathematisches Institut der Ludwig-Maximilians
  Universit\"at M\"unchen}, 2008.
\newblock \url{http://www.mathematik.uni-muenchen.de/~lfmoser/da.pdf}.

\bibitem[MV]{MorelVoevodsky}
Fabien Morel and Vladimir Voevodsky.
\newblock {${\bf A}\sp 1$}-homotopy theory of schemes.
\newblock {\em Inst. Hautes \'Etudes Sci. Publ. Math.}, (90):45--143 (2001),
  1999.

\bibitem[Nis]{Nisnevich1}
Yevsey Nisnevich.
\newblock Espaces homog\`enes principaux rationnellement triviaux et
  arithm\'etique des sch\'emas en groupes r\'eductifs sur les anneaux de
  {D}edekind.
\newblock {\em C. R. Acad. Sci. Paris S\'er. I Math.}, 299(1):5--8, 1984.

\bibitem[Pan1]{PaninPurity2010}
Ivan Panin.
\newblock Purity conjecture for reductive groups.
\newblock {\em Vestnik St. Petersburg Univ. Math.}, 43(1):44--48, 2010.

\bibitem[Pan2]{PaninNiceTriples}
Ivan Panin.
\newblock {Nice triples and the Grothendieck--Serre conjecture concerning
  principal G-bundles over reductive group schemes}.
\newblock {\em Duke Math. Journal}, 168(2):351--375, 2019.

\bibitem[Pan3]{PaninFiniteFieldsIzvestiya}
Ivan~Alexandrovich Panin.
\newblock Proof of the {G}rothendieck--{S}erre conjecture on principal bundles
  over regular local rings containing a field.
\newblock {\em Izvestiya: Mathematics}, 84(4):780, 2020.

\bibitem[Pan4]{PaninPurity17}
Ivan~Alexandrovich Panin.
\newblock Two purity theorems and the {G}rothendieck--{S}erre conjecture
  concerning principal-bundles.
\newblock {\em Sbornik: Mathematics}, 211(12):1777--1794, 2020.

\bibitem[Pop]{Popescu}
Dorin Popescu.
\newblock General {N}\'eron desingularization and approximation.
\newblock {\em Nagoya Math. J.}, 104:85--115, 1986.

\bibitem[PSV]{PaninStavrovaVavilov}
Ivan Panin, Anastasia Stavrova, and Nikolai Vavilov.
\newblock On {G}rothendieck-{S}erre's conjecture concerning principal
  {$G$}-bundles over reductive group schemes: {I}.
\newblock {\em Compos. Math.}, 151(3):535--567, 2015.

\bibitem[RR]{RagunathanRamanthan}
Madabusi~S. Raghunathan and Annamalai Ramanathan.
\newblock Principal bundles on the affine line.
\newblock {\em Proc. Indian Acad. Sci. Math. Sci.}, 93(2-3):137--145, 1984.

\bibitem[SGA]{SGA4-2}
{\em Th\'eorie des topos et cohomologie \'etale des sch\'emas. {T}ome 2}.
\newblock Lecture Notes in Mathematics, Vol. 270. Springer-Verlag, Berlin-New
  York, 1972.
\newblock S{\'e}minaire de G{\'e}om{\'e}trie Alg{\'e}brique du Bois-Marie
  1963--1964 (SGA 4), Dirig{\'e} par M. Artin, A. Grothendieck et J. L.
  Verdier. Avec la collaboration de N. Bourbaki, P. Deligne et B. Saint-Donat.

\bibitem[Sor]{SorgerLecturesBundles}
Christoph Sorger.
\newblock Lectures on moduli of principal {$G$}-bundles over algebraic curves.
\newblock In {\em School on {A}lgebraic {G}eometry ({T}rieste, 1999)}, volume~1
  of {\em ICTP Lect. Notes}, pages 1--57. Abdus Salam Int. Cent. Theoret.
  Phys., Trieste, 2000.

\bibitem[Spi]{SpivakovskyPopescu}
Mark Spivakovsky.
\newblock A new proof of {D}. {P}opescu's theorem on smoothing of ring
  homomorphisms.
\newblock {\em J. Amer. Math. Soc.}, 12(2):381--444, 1999.

\bibitem[Swa]{SwanOnPopescu}
Richard~G. Swan.
\newblock N\'eron-{P}opescu desingularization.
\newblock In {\em Algebra and geometry ({T}aipei, 1995)}, volume~2 of {\em
  Lect. Algebra Geom.}, pages 135--192. Int. Press, Cambridge, MA, 1998.

\bibitem[Tho]{ThomasonResolution}
Robert~W. Thomason.
\newblock Equivariant resolution, linearization, and {H}ilbert's fourteenth
  problem over arbitrary base schemes.
\newblock {\em Adv. in Math.}, 65(1):16--34, 1987.

\end{thebibliography}

\end{document}